\documentclass[10pt]{amsart}
\usepackage{amsmath,amsthm,amsfonts,amssymb,amscd,flafter,pinlabel}
\usepackage{mathtools}
\usepackage[mathscr]{eucal}
\usepackage{graphics}
 \usepackage[all]{xy}
\usepackage{caption, subcaption}
\usepackage[usenames,dvipsnames]{color}
\usepackage{graphicx}

\usepackage{pgf}
\usepackage{tikz}
\usetikzlibrary{arrows,automata}
\usepackage[latin1]{inputenc}

\usepackage[colorlinks=true, urlcolor=NavyBlue, linkcolor=NavyBlue, citecolor=NavyBlue, pdftitle={Lecture notes on trisections and chomology}, pdfauthor={Peter Lambert-Cole}, pdfsubject={trisections}, pdfkeywords={4-manifolds}]{hyperref}

\newtheorem{theorem}{Theorem}[section]
\newtheorem{corollary}[theorem]{Corollary}

\newtheorem{question}[theorem]{Question}
\newtheorem{proposition}[theorem]{Proposition}

\newtheorem{lemma}[theorem]{Lemma}

\theoremstyle{definition}
\newtheorem{definition}[theorem]{Definition}

\theoremstyle{definition}

\newcommand{\Z}{\ensuremath{\mathbb{Z}}}

\newcommand\alphas{\mbox{\boldmath$\alpha$}}
\newcommand\betas{\mbox{\boldmath$\beta$}}
\newcommand\gammas{\mbox{\boldmath$\gamma$}}

\def\s{\mathfrak{s}}

\newcommand{\cM}{\mathcal{M}}

\newcommand{\CC}{\mathbb{C}}
\newcommand{\RR}{\mathbb{R}}
\newcommand{\del}{\partial}

\newcommand{\ZZ}{\mathbb{Z}}

\newcommand{\cF}{\mathcal{F}}
\newcommand{\cT}{\mathcal{T}}

\newcommand{\cC}{\mathcal{C}}
\newcommand{\cK}{\mathcal{K}}

\newcommand{\cZ}{\mathcal{Z}}

\newcommand{\CP}{\mathbb{CP}}

\newcommand{\Hch}{\check{H}}
\newcommand{\Cch}{\check{C}}
\newcommand{\cH}{\mathcal{H}}
\newcommand{\cDR}{\mathcal{DR}}

\newcommand{\SpinC}{\text{Spin}^{\CC}}

\begin{document}

\title[{Lecture notes on trisections and cohomology}]{Lecture notes on trisections and cohomology}

\author[P. Lambert-Cole]{Peter Lambert-Cole}
\address{Department of Mathematics \\ University of Georgia}
\email{plc@uga.edu}

\maketitle


\begin{abstract}

These notes are from the first half of a seminar on symplectic trisections at the Max Planck Institute for Mathematics in Spring 2020.

\end{abstract}

\section{Introduction}

A motivating question in 4-manifold topology is

\begin{question}
To what extent are general 4-manifolds similar to projective complex surfaces?
\end{question}

Donaldson showed that, like projective surfaces, every closed symplectic manifold admits a Lefschetz pencil \cite{Donaldson}.  Later, Auroux, Donaldson and Katzarkov showed that near-symplectic manifolds admit so-called broken Lefschetz pencils\footnote{The term `singular Lefschetz pencil' was used in \cite{ADK}}\cite{ADK}.  Baykur then proved that every closed, oriented smooth 4-manifold admits a broken Lefschetz fibration over $S^2$ \cite{Baykur}.  This gives one sense in which all such 4-manifolds are similar to projective surfaces.

It is a classical fact, known as Theorem B, that over a Stein domain, coherent sheaves have no higher cohomology.  That is, if $Z$ is Stein and $\cF$ is a coherent sheaf, then $H^i(Z;\cF) = 0$ for $i > 0$.  A consequence is that if $X$ is a complex manifold, $\cF$ is a coherent sheaf, and $\cZ = \{Z_i\}$ is an open cover of $X$ by Stein domains, then the sheaf cohomology of $\cF$ can be computed by the Cech complex with respect to the open cover $\cZ$:
\[H^*(X;\cF) \cong \Hch^*(\cZ;\cF).\]
On a projective surface, Hodge theory implies that Dolbeault cohomology refines de Rham cohomology.  Specifically, there is an isomorphism
\[H^k(X;\CC) \cong \bigoplus_{i + j = k} H^{i,j}_{\overline{\partial}}(X;\CC)\]
In addition, Dolbeault's Theorem states that Dolbeault cohomology is isomorphic to the cohomology of the sheaf of holomorphic differential forms:
\[H^{i,j}_{\overline{\partial}}(X;\CC) \cong H^i(X;\Omega^j).\]
Moreover, applying Serre duality to the constant sheaf $\underline{\CC}$ shows that there is an isomorphism
\[H^{i,j}_{\overline{\partial}}(X;\CC) \cong H^{n-i,n-j}_{\overline{\partial}}(X;\CC)\]
where $n$ is the complex dimension of $X$.

Interestingly, trisections of 4-manifolds reveal similar results for singular and de Rham cohomology. The four-dimensional handlebody $\natural_k S^1 \times B^3$ admits a Stein structure.  Thus, since every closed 4-manifold admits a trisection, it can be covered by three domains that admit Stein structures.  In addition, by slightly enlarging the sectors of trisection, we get an open cover $\cT = \{U_1,U_2,U_3\}$, where
\begin{enumerate}
\item $U_i$ is diffeomorphic to $\natural_{k_i} S^1 \times B^3$,
\item $U_i \cap U_j$ is diffeomorphic to $\natural_{g} S^1 \times B^3$, and
\item $U_1 \cap U_2 \cap U_3$ is diffeomorphic to $\Sigma_g \times D^2$.
\end{enumerate}
Let $\cC^i$ denote the presheaf on $X$ defined as
\[\cC^i(U) \coloneqq H^i(U;\ZZ)\]
It is clear that $\cC^i$ is a presheaf.  However, in general it is not a sheaf as it satisfies the gluing axiom but not the locality axiom.  In particular, it is not separated.  Nonetheless, we can compute the Cech cohomology $\Hch^*(\cT,\cC^i)$ of the presheaf $\cC^i$ with respect to the open cover $\cT$.  

Methods to compute the homology of 4-manifolds from a trisection have been given by Feller, Klug, Schirmer and Zemke \cite{FKSZ} and by Florens and Moussard \cite{FM}.  Reinterpreting their results, we get the following theorems:

\begin{theorem}[Hodge/Dolbeault Theorem]
\label{thrm:trisection-Hodge}
There is an isomorphism
\[H^{k}(X;\ZZ) \cong \bigoplus_{i + j  = k} \Hch^i(\cT,\cC^j)\]
Moreover, we have the following `Hodge Diamond' for the cohomology of a trisected 4-manifold
\[
\begin{array}{c c c c c}
& & H^4(X;\ZZ) & & \\
\\
& 0 & & H^3(X;\ZZ) & \\
\\
0 & & H^2(X;\ZZ) & & 0 \\
\\
& H^1(X;\ZZ) & & 0 & \\
\\
& & H^0(X;\ZZ) & &
\end{array}\]
\end{theorem}

In particular, the Cech complex $\Cch^*(\cT,\cC^1)$ -- representing the middle diagonal of the Hodge diamond -- is essentially given in \cite[Section 2.1]{FM} but not described as such.

We can also interpret the symmetry of the Hodge diamond as Serre duality.

\begin{theorem}[Serre duality]
\label{thrm:trisection-Serre}
There is an isomorphism
\[ \Hch^i(\cT,\cC^j) \otimes \RR \cong \Hch^{2-i}(\cT,\cC^{2-j}) \otimes \RR\]
\end{theorem}

\subsection{$2^{\text{nd}}$ Cohomology as (1,1)-classes}

By analogy with complex geometry, we refer to any class in $\Hch^1(\cT,\cC^1) \cong H^2(X;\ZZ)$ as a {\it $(1,1)$-class}.  On a projective surface, the Lefschetz theorem states that the integral (1,1) classes are precisely those that can be represented by a divisor.  The proof of Theorem \ref{thrm:trisection-Hodge} further implies that every class of is a (1,1) class.

\begin{theorem}
\label{thrm:11}
Every class in $H^2(X;\ZZ)$ is a (1,1)-class with respect to the trisection $\cT$.  Specifically
\[H^2(X;\ZZ) \cong \Hch^1(\cT,\cC^1)\]
\end{theorem}

Unpacking the definition of Cech cohomology, this means that every element of $H^2(X)$ is represented by a triple $(\beta_1,\beta_2,\beta_3)$ where $\beta_{\lambda}$ is a 1-dimensional cohomology class on the handlebody $H_{\lambda}$ of the trisection.  We will describe several geometric interpretations of this.

\begin{enumerate}
\item {\bf DeRham Cohomology} Every class $\omega \in H^2_{DR}(X)$ can be represented by a triple $(\beta_1,\beta_2,\beta_3)$ where $\beta_{\lambda}$ is a closed 1-form on $H_{\lambda}$.
\item {\bf $\CC$-bundles}.  Recall that isomorphism classes of $\CC$-line bundles over $X$ are classified by $H^2(X;\ZZ)$ and homotopy classes of maps from $H_{\lambda}$ to $S^1$ are classified by $H^1(H_{\lambda};\ZZ)$.  Take a line bundle $E$ with $1^{\text{st}}$-Chern class $c_1(E)$.  Then $E$ can be trivialized over each sector $Z_{\lambda}$ of the trisection and the triple $(\beta_1,\beta_2,\beta_3)$ corresponding to $c_1(E)$ determines the transition maps (up to homotopy).
\item {\bf $\SpinC$-structures}.  The set of $\SpinC$-structures on $X$ is an affine copy of $H^2(X;\ZZ)$.  Following Gompf, we show how to interpret a $\SpinC$-structure as an almost-complex structure on the spine of the trisection.  Then, the action of $H^2(X;\ZZ)$ can be described in terms of `Lutz twists' along a collection curves representing homology classes in $H_1(H_{\lambda})$ that are hom-dual to $(\beta_1,\beta_2,\beta_3)$.
\end{enumerate}

\vfill

\pagebreak

\section{Singular cohomology}

Let $X = Z_1 \cup Z_2 \cup Z_3$ be a trisection of $X$, let $Y_{\lambda} = \del Z_{\lambda}$ and let $H_{\lambda} = \Z_{\lambda-1} \cap Z_{\lambda}$.  Let $\Sigma$ be the central surface.  The inclusion 
\[\iota_{\lambda}: \Sigma \rightarrow H_{\lambda}\]
induces two maps
\begin{align*}
(\iota_{\lambda})_*&: H_1(\Sigma) \rightarrow H_1(H_{\lambda}) &
(\iota_{\lambda})^*&: H^1(H_{\lambda}) \rightarrow H^1(\Sigma)
\end{align*}
Define subspaces
\begin{align*}
L_{\lambda} &\coloneqq \text{ker}((\iota_{\lambda})_*) \subset H_1(\Sigma) &
M_{\lambda} &\coloneqq \text{Im}((\iota_{\lambda})^*) \subset H^1(\Sigma)
\end{align*}
We can use the intersection pairing $\langle -, - \rangle_{\Sigma}$ on $H_1(\Sigma)$ to define an isomorphsm $\pi: H_1(\Sigma) \rightarrow H^1(\Sigma)$ by setting
\[\pi(x) = \langle -, x \rangle_{\Sigma}\]
Furthermore, we have inclusion maps $\kappa_{i,j}: H_j \hookrightarrow Y_i$ and $\rho_i: Y_i \rightarrow Z_i$ for $i = 1,2,3$ and $j = i-1,i$.   These induce maps
\begin{align*}
(\kappa_{i,j})_*&: H_1(H_j) \rightarrow H_1(Y_i) & (\rho_i)_*&: H_1(Y_i) \rightarrow H_1(Z_i) \\
(\kappa_{i,j})^*&: H^1(Y_i) \rightarrow H^1(H_j) & (\rho_i)^* &: H^1(Z_i) \rightarrow H^1(Y_i)
\end{align*}

\subsection{Hodge Diamond}

The results in \cite{FKSZ,FM} compute homology.  In particular, we have the following expression for $H_*(X)$.

\begin{theorem}[\cite{FM}]
The homology of $X$ with $\ZZ$-coefficients is the homology of the complex
\[
\xymatrix{
0 \ar[r] & \ZZ \ar[r]^-{0} & (L_{1} \cap L_{2}) \oplus (L_{2} \cap L_{3}) \oplus (L_{3} \cap L_{1}) \ar[r]^-{\zeta} & L_{1} \oplus L_{2} \oplus L_{3} \ar[r]^-{\iota} & H_1(\Sigma) \ar[r]^0 &\ZZ \rightarrow 0 }
\]
where $\zeta(a,b,c) = (c - a, a - b, b - c)$ and $\iota(a,b,c) = a + b + c$.
\end{theorem}

The middle terms of this complex are essentially the Cech complex.

\begin{proposition}
\label{prop:middle-isom}
There is a chain complex isomorphism
\[ 
\xymatrix{
0 \ar[r] \ar[d] ^{0} & (L_{1} \cap L_{2}) \oplus (L_{2} \cap L_{3}) \oplus (L_{3} \cap L_{\alpha}) \ar[r]^-{\zeta} \ar[d]^{\phi_1} & L_{\alpha} \oplus L_{2} \oplus L_{3} \ar[r]^-{\iota} \ar[d]^{\phi_2} & H_1(\Sigma) \ar[d]^-{\pi} \ar[r] & 0  \ar[d]^{0} \\
0 \ar[r] & \bigoplus_{\lambda} H^1(Z_{\lambda}) \ar[r]^-{\delta_1} & \bigoplus_{\lambda} H^1(H_{\lambda}) \ar[r]^-{\delta_2} & H^1(\Sigma) \ar[r] & 0
}
\]
\end{proposition}

The second complex of this proposition is exactly the Cech complex of $\cC^1$ with respect to $\cT$, thus by applying Poincare Duality we obtain the following corollary.

\begin{corollary}
for $i = 1,2,3$, there are isomorphisms
\[H_{4-i}(X;\ZZ) \cong H^i(X;\ZZ) \cong \Hch^{i-1}(\cT,\cC^1)\]
\end{corollary}

\begin{proof}[Proof of Proposition \ref{prop:middle-isom}]
By definition, $Z_{\lambda} = \natural_{k_{\lambda}} S^1 \times B^3$ and $Y_{\lambda} = \del Z_{\lambda} = \#_{k_{\lambda}} S^1 \times S^2$.  In particular
\[H_1(Z_{\lambda}) \cong H_1(Y_{\lambda}) \cong \ZZ^{k_{\lambda}}\]
We can apply the Mayer-Vietoris sequence to the Heegaard splitting $Y_{\lambda} = H_{\lambda} \cup H_{\lambda+1}$ to get the sequence
\[\rightarrow H_2(H_{\lambda}) \oplus H_2(H_{\lambda+1}) \rightarrow H_2(Y_{\lambda}) \rightarrow H_1(\Sigma) \rightarrow H_1(H_{\lambda}) \oplus H_1(H_{\lambda+1}) \rightarrow H_1(Y_{\lambda}) \rightarrow H_0(\Sigma)\]
Since $H_2(H_{\lambda}) = H_2(H_{\lambda+1}) = 0$, we see that
\[H^1(Y_{\lambda}) \cong H_2(Y_{\lambda}) \cong \text{ker}\left( H_1(\Sigma) \rightarrow  H_1(H_{\lambda}) \oplus H_1(H_{\lambda+1}) \right) \cong L_{\lambda} \cap L_{\lambda+1}\]
where the first isomorphism follows by Poincare duality.  This defines $\phi_1$.

Using the long exact sequence of the pair $(H_{\lambda},\Sigma)$ we obtain
\[ H_2(H_{\lambda}) \rightarrow H_2(H_{\lambda},\Sigma) \rightarrow H_1(\Sigma) \rightarrow H_1(H_{\lambda}) \rightarrow \]
Since $H_2(H_{\lambda}) = 0$, we see that
\[H^1(H_{\lambda}) \cong  H_2(H_{\lambda},\Sigma) \cong \text{ker}(H_1(\Sigma) \rightarrow H_1(H_{\lambda})) = L_{\lambda}\]
This defines $\phi_2$.
\end{proof}

The remaining cohomology groups are straightforward to calculate.

\begin{proposition}
The cohomology groups of $\cH^0$ are
\begin{align*}
\Hch^0(\cT,\cH^0) & \cong \ZZ \\
\Hch^1(\cT,\cH^0) & \cong 0 \\
\Hch^2(\cT,\cH^0) & \cong 0 \\
\end{align*}
\end{proposition}

\begin{proof}
Each open set $U_i$ and each double and triple intersection is connected and so
\[ H^0(U_i;\ZZ) \cong H^0(U_i \cap U_j; \ZZ) \cong H^0(U_1 \cap U_2 \cap U_3) \cong \ZZ\]
The Cech complex is therefore
\[0 \rightarrow \ZZ^3 \rightarrow \ZZ^3 \rightarrow \ZZ \rightarrow 0\]
If $\{a,b,c\}$ is a chain in $\Cch^0(\cT,\cH^0)$ then
\[\delta_0 \{a,b,c\} = \{a - b, b - c, c - a\}\]
Thus, this chain is coclosed if and only if $a = b = c$.  Thus, $\Hch^0(\cT,\cH^0)  \cong \ZZ \langle \{a,a,a\} \rangle \cong \ZZ$.
If $\{a,b,c\}$ is a chain in $\Cch^1(\cT,\cH^0)$, then
\[ \delta_1 \{a,b,c\} = \{a + b + c\}\]
The chain is coclosed if and only if it has the form $\{a,b,-a-b\} = a\{1,0,-1\} + b\{0,1,-1\}$.  Both elements $\{1,0,-1\}$ and $\{0,1,-1\}$ are in the image of $\delta_0$, so $\Hch^1(\cT,\cH^0) \cong 0$.  Finally, the differential $\delta_1$ is surjective so $\Hch^1(\cT,\cH^0) \cong 0$ as well.
\end{proof}

\begin{proposition}
The cohomology groups of $\cH^2$ are
\begin{align*}
\Hch^0(\cT,\cH^2) & \cong 0\\
\Hch^1(\cT,\cH^2) & \cong 0 \\
\Hch^2(\cT,\cH^2) & \cong \ZZ \\
\end{align*}
\end{proposition}

\begin{proof}
Each $U_i$ and each double intersection $U_i \cap U_j$ is a four-dimensional 1-handlebody.  Thus
\[H^2(U_i;\ZZ) \cong H^2(U_i \cap U_j;\ZZ) \cong 0\]
The Cech complex is therefore
\[ 0 \rightarrow 0 \rightarrow 0 \rightarrow \ZZ \rightarrow 0\]
and the proposition follows immediately.
\end{proof}

\vfill

\pagebreak

\section{deRham}

Let $\cDR^i$ denote the presheaf on $X$ defined as 
\[\cDR^i(U) \coloneqq H^i_{DR}(U;\RR)\]

\subsection{DeRham to Cech isomorphism}

\begin{theorem}
There are isomorphisms
\begin{align*}
H^1_{DR}(X;\RR) &\cong \Hch^0(\cT,\cDR^1) & H^0_{DR}(X;\RR) &\cong \Hch^0(\cT,\cDR^0)\\
H^2_{DR}(X;\RR) &\cong \Hch^1(\cT,\cDR^1) & H^4_{DR}(X;\RR) &\cong \Hch^2(\cT,\cDR^2)\\
H^3_{DR}(X;\RR) &\cong \Hch^2(\cT,\cDR^1) 
\end{align*}
\end{theorem}

We break up the proof by the degree of the cohomology group:

\noindent {\bf Degree 0}:  The cohomology group $H^0_{DR}(X;\RR)$ consists of constant functions.  Given a constant function $C: X \rightarrow \RR$, its restriction to $U_{\lambda}$ is also a constant function $C: U_{\lambda} \rightarrow \RR$ and therefore an element of $H^0_{DR}(U_{\lambda};\RR)$.  The isomorphism from deRham to Cech is given by $C \mapsto (C,C,C)$.

Conversely, an element of $\Hch^0(\cT,\cDR^0)$ is a triple $(C_1,C_2,C_3)$ of constant functions whose restrictions to the pairwise intersections agree.  In other words, $C_1 = C_2 = C_3 = C$.  The inverse isomorphism is therefore $(C,C,C) \mapsto C$. \\

\noindent {\bf Degree 1}: The map from DeRham to Cech is identical to the degree 0 case above.  Given some closed 1-form $\beta$, the corresponding element in Cech cohomology is given by restricting $\beta$ to each $U_{\lambda}$.  

The inverse isomorphism is more complicated.  In particular, an element of $\Hch^0(\cT,\cDR^1)$ is a triple $([\beta_1],[\beta_2],[\beta_3])$ of cohomology classes, not specific closed forms.  Choose representative closed 1-forms $\beta_1,\beta_2,\beta_3$.  By assumption, the restrictions satisfy
\[ [\beta_{\lambda-1}] = [\beta_{\lambda}] \in H^1_{DR}(U_{\lambda-1} \cap U_{\lambda};\RR)\]
Therefore, $\beta_{\lambda} - \beta_{\lambda-1} = dg_{\lambda}$ for some function $g: U_{\lambda - 1} \cap U_{\lambda} \rightarrow \RR$.

\noindent {\bf Exercise:} Show that there exist functions $f_{\lambda}: U_{\lambda} \rightarrow \RR$ such that on $U_{\lambda -1 } \cap U_{\lambda}$
\[ \beta_{\lambda -1} + df_{\lambda - 1} = \beta_{\lambda} + df_{\lambda}\]

Consequently, we can represent our original Cech class by the triple $(\beta_1 + df_1, \beta_2 + df_2, \beta_3 + df_3)$ and these 1-forms glue into a global 1-form $\beta$. \\

\noindent {\bf Degree 2}: In this case, the maps in both directions are more complicated and we need to check that they are in fact isomorphisms.  First, choose a class $[\omega] \in H^2_{DR}(X;\RR)$ and represent it by a closed 2-form $\omega$.  The restriction $\omega|_{U_{\lambda}}$ is exact since $H^2_{DR}(U_{\lambda};\RR) = 0$, thus we can choose a primitive $\alpha_{\lambda}$ for $\omega|_{U_{\lambda}}$.  Over the double intersection $U_{\lambda-1} \cap U_{\lambda}$, the restrictions $\alpha_{\lambda - 1}$ and $\alpha_{\lambda}$ are both primitives for $\omega$, therefore their difference $\alpha_{\lambda} - \alpha_{\lambda-1}$ is closed.  Consequently, the map from DeRham to Cech is given by
\[ \omega \mapsto (\alpha_1 - \alpha_3, \alpha_2 - \alpha_1, \alpha_3 - \alpha_2)\]
There were three sources of indeterminancy:
\begin{enumerate}
\item we could replace $\alpha_{\lambda}$ by $\alpha_{\lambda} + df_{\lambda}$ for some function $f_{\lambda}: U_{\lambda} \rightarrow \RR$,
\item we could replace $\omega$ by $\omega + d\mu$ for some global 1-form $\mu$, and
\item we could replace the primitive $\alpha_{\lambda}$ with $\alpha_{\lambda} + \rho_{\lambda}$, where $\rho$ is a closed 1-form on $U_{\lambda}$
\end{enumerate}

\noindent {\bf Exercise}: 

\begin{enumerate}
\item Show that modifying the primitives $\{\alpha_{\lambda}\}$ by exact 1-forms results in the same Cech cochain.
\item Show that we can choose primitives for $\omega + d\mu$ that result in the same Cech cochain 
\item Show that modifying the primitives $\{\alpha_{\lambda}\}$ by closed 1-forms $\{\rho_{\lambda}\}$ changes the Cech cochain by a Cech coboundary.
\end{enumerate}

Conversely,  given a class in $\Hch^0(\cT,\cDR^1)$, choose a fixed cochain $([\beta_1],[\beta_2],[\beta_3])$ and fixed closed 1-forms $\{\beta_1,\beta_2,\beta_3\}$ to represent this class.

\noindent {\bf Exercise:} 
\begin{enumerate}
\item There exists a triple of 1-forms $\{\alpha_{\lambda}\}$ on the open sets $\{U_{\lambda}\}$ such that $\alpha_{\lambda} - \alpha_{\lambda - 1} = \beta_{\lambda}$.
\item The 2-forms $\{d\alpha_1,d\alpha_2,d\alpha_3\}$ glue together to give a global 2-form $\omega$.
\item Modifying the choices -- modifying the Cech cochain by a coboundary, modifying the closed 1-forms $\{\beta_{\lambda}\}$ by exact 1-forms, modifying the choices of $\{\alpha_{\lambda}\}$ --- results in a cohomologous 2-form $\omega'$.
\end{enumerate}

\noindent {\bf Degree 3}

Given a class $[\mu] \in H^3_{DR}(X;\RR)$, represent it by a closed 3-form $\mu$.  Since $H^3_{DR}(U_{\lambda};\RR) = 0$, we can choose a primitive $\omega_{\lambda}$ for $\mu$ over each $U_{\lambda}$.  The differences $\omega_{\lambda} - \omega_{\lambda-1}$ are closed and represent elements of $H^3_{DR}(U_{\lambda} \cap U_{\lambda-1};\RR) = 0$.  In particular, these forms are also exact and we can choose further primitives 1-forms$\{\beta_{\lambda}\}$.  Restricting to the triple intersection $U_1 \cap U_2 \cap U_3$ we get 1-form $\beta = \beta_1 + \beta_2 + \beta_3$ that is closed since
\[d\beta = d \beta_1 + d \beta_2 + d\beta_3 = (\omega_1 - \omega_3) + (\omega_2 - \omega_1) + (\omega_3 - \omega_2) = 0.\]
Thus, $[\mu]$ is sent to an element $[\beta] \in H^1_{DR}(U_1 \cap U_2 \cap U_3;\RR)$ and therefore represents a Cech 2-cocycle.

\noindent {\bf Exercise:}
\begin{enumerate}
\item Show that changing $\omega_{\lambda}$ by a closed 2-form results in the same Cech 2-cocycle
\item Show that changing $\beta_{\lambda}$ by a closed 1-form modifies the resulting Cech 2-cocycle by a Cech 2-coboundary.
\end{enumerate}

The inverse map can be constructed by an argument similar to the Degree 2 case; we leave it as an exercise.

\noindent {\bf Exercise}: Construct the inverse map $\Hch^2(\cT,\cC^1) \rightarrow H^3_{DR}(X;\RR)$ and show that it is well-defined. \\

\noindent {\bf Degree 4}:  The isomorphism is constructed in a analogous method to the Degree 3 case and we leave it as an exercise to the reader.

\noindent {\bf Exercise}: Construct the isomorphism $H^4_{DR}(X;\RR) \cong \Hch^2(\cT,\cC^2)$.

\subsection{Intersection Pairing}

The intersection pairing on DeRham cohomology can also be expressed in terms of the Cech cohomology of the DeRham presheafs.  In particular, we can describe the following pairings
\begin{align*}
H^2_{DR}(X) \times H^2_{DR}(X) \rightarrow \RR \\
H^3_{DR}(X) \times H^1_{DR}(X) \rightarrow \RR \\
\end{align*}
Moreover, we can describe the pairing obtained by integrating a closed $p$-form over a closed $p$-dimensional submanifold.
\begin{align*}
H^2_{DR}(X) \times H_2(X;\ZZ) \rightarrow \RR \\
H^3_{DR}(X) \times H_3(X;ZZ) \rightarrow \RR \\
H^4_{DR}(X) \times H_4(X;\ZZ) \rightarrow \RR
\end{align*}

\begin{theorem}[Intersection Pairing]
Let $X$ be a trisected 4-manifold.
\begin{enumerate}
\item Let $\omega_1,\omega_2$ be a pair of closed 2-forms.  Suppose that under the DeRham-Cech isomorphism we have
\[ [\omega_1] \mapsto (\alpha_1,\alpha_2,\alpha_3) \qquad [\omega_2] \mapsto (\beta_1,\beta_2,\beta_3)\]
Then
\[\int_X \omega \wedge \mu = \int_{\Sigma} \alpha_1 \wedge \beta_2 = \int_{\Sigma} \alpha_2 \wedge \beta_3 = \int_{\Sigma} \alpha_3 \wedge \beta_1\]
\item Let $\mu$ be a closed 3-form and $\alpha$ be a closed 1-form.  Suppose that under the DeRham-Cech isomorphism we have that $[\mu] \mapsto [\beta]$.  Then
\[\int_X \mu \wedge \alpha = \int_{\Sigma} \beta \wedge \alpha|_{\Sigma}\]
\end{enumerate}
\end{theorem}

\noindent {\bf Exercise:} Prove these statements (Hint: Use Stokes's Theorem combined with the arguments in the previous subsection) \\

To describe the integration pairing, we first fix some notation.
\begin{enumerate}
\item Let $\cK$ be an embedded, oriented closed surface in general position with respect to the trisection.  Let $\tau^{\cK}_{\lambda}$ denote the tangle $\cK \pitchfork H_{\lambda}$.  We orient $\tau_{\lambda}$ as follows: since $\cK$ is oriented, the intersection $F_{\lambda} = \cK \cap Z_{\lambda}$ is oriented.  The boundary $\del F_{\lambda}$ inherits an orientation from $F_{\lambda}$; the tangle $\tau^{\cK}_{\lambda}$ is a subset of this boundary and inherits an orientation.
\item Let $\cM$ be an embedded, oriented, closed hypersurface in general position with respect to the trisection.  In particular, the intersection $\cM \pitchfork \Sigma$ is a simple closed curve $\gamma_{\cM}$.
\end{enumerate}

\begin{theorem}[Integration Pairing]
Let $X$ be a trisected 4-manifold.
\begin{enumerate}
\item Let $\omega$ be a closed 2-form on $X$ that maps to $(\beta_1,\beta_2,\beta_3)$ under the DeRham-Cech isomorphism and let $\cK$ be an embedded, oriented closed surface.  Then
\[ \int_{\cK} \omega = \sum_{\lambda = 1,2,3} \int_{\tau^{\cK}_{\lambda}} \beta_{\lambda}\]
\item Let $\mu$ be a closed 3-form on $X$ that maps to $\beta \in H^1_{DR}(\Sigma)$ under the DeRham-Cech isomorphism and let $\cM$ be an embedded, oriented, closed hypersurface.  Then
\[ \int_{\cM} \mu = \int_{\gamma_{\cM}} \beta\]
\item Let $\Omega$ be a closed 4-form on $X$ that maps to $\omega \in H^2_{DR}(\Sigma)$ under the DeRham-Cech isomorphism.  Then
\[ \int_X \Omega = \int_{\Sigma} \omega\]
\end{enumerate}
\end{theorem}

\noindent {\bf Exercise:} Prove these statements (Hint: Again, use Stokes's Theorem).

\section{Complex Line Bundles}

\subsection{Algebraic Topology}
First, we recall some facts from algebraic topology.
\begin{enumerate}
\item The circle $S^1$ is a $K(\ZZ,1)$.  In particular, there is a 1-1 correspondence between classes in $H^1(X;\ZZ)$ and homotopy classes of maps $f: X \rightarrow S^1$.
\item The space $\CP^{\infty}$ is a $K(\ZZ,2)$.  In particular, there is a 1-1 correspondence between classes in $H^2(X;\ZZ)$ and homotopy classes of maps $f: X \rightarrow \CP^{\infty}$.  The cohomology ring of $\CP^{\infty}$ is $\ZZ[\alpha]$, where $\alpha$ has degree 2, and the identification between maps and cohomology classes is given by
\[f \leftrightarrow f^*(\alpha)\]
\item The space $\CP^{\infty}$ is the classifying space for $U(1)$ (equivalently $\CC$-line) bundles.  In particular, there is a 1-1 correspondence between $\CC$-line bundles on $X$, up to isomorphism, and homotopy classes of maps $f: X \rightarrow \CP^{\infty}$.  There is a {\it tautological line bundle} $E \rightarrow \CP^{\infty}$ and the correspondence between maps and $\CC$-bundles is given by
\[ f \leftrightarrow f^*(E)\]
\item The {\it $1^{\text{st}}$-Chern class} is a complete invariant of $\CC$-line bundles and connects (2) and (3) above.  In particular, for the tautological bundle $E$ on $\CP^{\infty}$ we have
\[c_1(E) = \alpha\]
Moreover, since Chern classes are characteristic, they are natural with respect to pullbacks and therefore
\[c_1(f^*(E)) = f^*(c_1(E)) = f^*(\alpha).\]
\end{enumerate}

\subsection{Chern classes of line bundles}

Using a trisection $\cT$ of $X$, we can explicitly see the equivalence
\[ \{ \text{$\CC$-bundles on X}\}/\sim \simeq \Hch^1(\cT,\cC^1) \cong H^2(X;\ZZ)\]

\noindent {\bf Line bundles to (1,1)-classes}: Take a line bundle $E$ on $X$.  Since each sector $Z_{\lambda}$ of a trisection is a 1-handlebody, we can choose a trivialization $s_{\lambda}$ of $E$ over $Z_{\lambda}$.  Up to homotopy, the potential choices of trivializations are in 1-1 correspondence with elements of $H^1(Z_{\lambda};\ZZ) \cong \ZZ^{k_{\lambda}}$.  Over the double intersection $H_{\lambda}$, we have two trivializations $s_{\lambda-1},s_{\lambda}$.  Taking their quotient, we obtain a map
\[g_{\lambda} \coloneqq \frac{s_{\lambda}}{s_{\lambda-1}} \rightarrow \CC^*\]
Composing this with the homotopy equivalence $\CC^* \simeq S^1$, the map $g_{\lambda}$ determines a homotopy class of maps from $H_{\lambda}$ to $S^1$.  In other words, the transition function $g_{\lambda}$ determines a unique element $\beta_{\lambda}$ of $H^1(H_{\lambda};\ZZ)$.  Moreover, since
\[g_1 g_2 g_3 = \frac{s_1}{s_3} \frac{s_2}{s_1} \frac{s_3}{s_2} = 1\]
the resulting triple $(\beta_1,\beta_2,\beta_3)$ is a Cech 1-cocycle in $\Cch^*(\cT,\cC^1)$.  Modifying the trivialization $s_{\lambda}$ by some element of $H^1(Z_{\lambda};\ZZ)$ changes the resulting cocycle by a Cech coboundary.  In particular, we obtain a well-defined element $c_1(E) \in \Hch^1(\cT,\cC^1)$.

\noindent {\bf (1,1)-classes to line bundles}:  Given a (1,1)-class $(\beta_1,\beta_2,\beta_3) \in  \Hch^1(\cT,\cC^1)$, we can represent $\beta_{\lambda} \in H^1(H_{\lambda};\ZZ)$ by a map $g_{\lambda}: H_{\lambda} \rightarrow S^1$.  Moreover, given the cocycle condition $\beta_1 + \beta_2 + \beta_3 = 0$ we can assume that $g_1g_2 g_3 = 1$.  In particular, the triple $\{g_1,g_2,g_3\}$ determines a triple of transition functions that allow us to construct a $\CC$-bundle over $X$.

\section{Almost-Complex Structures}

An {\it almost-complex structure} $J$ on $X$ is a fiberwise homomorphism $J: TX \rightarrow TX$ such that $J^2 = -I$.  This turns every fiber $T_xX$ into a complex vector space, where $J$ is multiplication by $i$.  Consequently, the almost-complex structure determines Chern classes $c_i(TX,J) \in H^{2i}(X;\ZZ)$.  The goal of this section is to describe almost-complex structures on the spine of a trisection.

\subsection{Field of complex tangencies}

Let $Y^3 \subset X^4$ be a smooth hypersurface and let $J$ be an almost-complex structure.  The {\it field of $J$complex tangencies} is defined to be
\[\xi \coloneq J(TY) \cap TY\]

\noindent {\bf Exercise} Show that $\xi$ has rank 2 at every point.  [Hint: $\xi_x$ is a $J$-complex line in $T_xX$].  In particular, $\xi$ is an {\it oriented} plane field.

\noindent {\bf Exercise} Let $\phi: X \rightarrow \RR$ be a function such that $Y = \phi^{-1}(0)$.  Show that the field of $J$-tangencies is the kernel of the 1-form $d^{\CC}\phi = d\phi(J-)$, restricted to $Y$.

\begin{proposition}
\label{prop:J-xi}
Let $Y$ be a 3-manifold.  Homotopy classes of almost-complex structures on $Y \times [0,1]$ are in 1-1 correspondence with homotopy classes of (coorientable) 2-plane fields on $Y$.
\end{proposition}

\begin{proof}
Let $J$ be an almost-complex structure on $Y \times [0,1]$ and let $\xi_t$ denote the field of $J$-tangencies along $Y \times \{t\}$.  It is immediately clear that $\{\xi_t\}$ is a homotopy of 2-plane fields.  Furthermore, let $J_s$ be a family of almost-complex structures and let $\xi_{s,t}$ denote the field of $J_s$-tangencies along $Y \times \{t\}$.  Again, this clearly gives a 2-parameter homotopy of plane fields on $Y$.

Now let $\xi$ be an oriented, coorientable 2-plane field and choose a fiberwise metric $g$ on $\xi$.  We can define an almost-complex structure $J: \xi \rightarrow \xi$ using the metric as follows.  Locally, we can choose an oriented, orthonormal frame $\{e_1,e_2\}$ and define
\[J(e_1) = e_2 \qquad J(e_2) = -e_1\]
and extend linearly.

\noindent {\bf Exercise} Show that, up to homotopy, this $J$ does not depend on the metric $g$ or the local orthonormal frame.

Next, let $\Lambda$ be an oriented line field that coorients $\xi$.  After choosing a metric $h$ on $\Lambda$, we obtain a unit-length section $\sigma$ of $\Lambda$ and can extend $J$ from $\xi$ to $TX$ by defining
\[J(\del_t) = \sigma \qquad J(\sigma) = - \del_t\]

\noindent {\bf Exercise} Show that, up to homotopy, this $J$ does not depend on the homotopy class of $J|_{\xi}$, the homotopy class of $\Lambda$, or the metric $h$.

Finally, we have to check that every $J$ on $Y \times [0,1]$ can be constructed in this way.  Choose some $J$ and define $E = \langle \del_t, J(\del_t) \rangle$ and $\Lambda = TY \cap E$.  Choose a nonvanishing section $\sigma$ of $\Lambda$.  Then
\[J(\del_t) = f \del_t + g \sigma\]
for some functions $f,g$.  By assumption $\{\del_t, J \del_t\}$ is an oriented basis for $E$ and therefore $g > 0$.  Since $J$ preserves $\xi$, we can define a family $J_s$ of almost-complex structures for $s \in [0,1]$ by defining
\[J_s|_{\xi} = J \qquad J_s(\del_t) = sf \del_t + g \sigma\]
After scaling the metric so that $|g \sigma| = 1$, we have that $J_0$ is almost-complex structure of the form constructed above and $J_1$ is our original $J$.
\end{proof}

\noindent {\bf Exercise}: $\Sigma \times D^2$ admits an almost-complex structure $J$ with $c_1(J) = 0$. [Hint: embed $\Sigma$ in $\CC^2$.]

\begin{lemma}
The spine of a trisection admits an almost-complex structure $J$.
\end{lemma}

\begin{proof}
By the previous exercise, we can choose some $J$ on a tubular neighborhood of the central surface $\Sigma$.  The remaining task is to extend it across each handlebody $H_{\lambda}$.  The almost-complex structure $J$ determines a hyperplane field $\xi_{\lambda}$ in a neighborhood of $\del H_{\lambda} = \Sigma$.  

\noindent {\bf Exercise}: Show that $\langle e(\xi_{\lambda}), [\Sigma] \rangle =\langle c_1(J), [\Sigma] \rangle =0$.  [Hint: choose a section $\sigma$ of $\xi_{\lambda}$ and a normal vector field $\nu$ to $H_{\lambda}$.  Then $\text{det}(\nu,\sigma) = 0$ precisely where $\sigma = 0$]

Consequently, it is possible to extend $\xi_{\lambda}$ across $H_{\lambda}$ and by Proposition \ref{prop:J-xi}, this determines a homotopy class of $J$ in a neighborhood of $H_{\lambda}$.
\end{proof}

\subsection{$1^{\text{st}}$-Chern class of $J$}.  Given some $J$ on the spine of a trisection, we can construct a 1-complex $C_J$ in the spine that represents the Poincare dual to $c_1(TX,J)$.

The central surface $\Sigma$ is canonically framed.  In particular, we can choose coordinates $(s,t)$ on $D^2$ such that pulling back the coordinates by the projection
\[\pi: \nu(\Sigma) \cong \Sigma \times D^2 \rightarrow D^2\]
we have that
\begin{align*}
\Sigma &= \pi^{-1}(0) & H_2 &= \pi^{-1}(0,t) \text{ for $t \geq 0$} \\
H_1 &= \pi^{-1}(s,0) \text{ for $s \leq 0$} & H_3 &= \pi^{-1}(-x,x) \text{ for $x \geq 0$}
\end{align*}

Consider the {\it conormal sequence} for the central surface $\Sigma$:
\[0 \rightarrow N^* \Sigma \rightarrow T^*X \rightarrow T^*\Sigma \rightarrow 0\]
A {\it coframing} of $\Sigma$ is a trivialization of its conormal bundle.  Since $N^*\Sigma$ is an $\RR^2$-bundle, a coframing is determined by a single, nowhere-vanishing section.  Moreover, it is clear from the conormal sequence that such a section is given by a nowhere-vanishing 1-form whose restriction to $\Sigma$ is identically 0.  An almost-complex structure $J$ determines a dual almost-complex structure $J^t: T^*X \rightarrow T^*X$.  Inserting this, we get a (nonexact) sequence
\[ \xymatrix{
N^*\Sigma \ar[r] & T^*X \ar[r]^-{J^t} & T^*X \ar[r] & T^*\Sigma
}\]
Given a section $\alpha$ of $N^*\Sigma$, we can push it through this sequence to get a 1-form $\widetilde{\alpha}$ on $\Sigma$, defined to be
\[\widetilde{\alpha} = \alpha(J-)|_{\Sigma}\]

\noindent {\bf Exercise}. A {\it complex point} of $\Sigma$ is a point $x \in \Sigma$ such that $J(T_x\Sigma) = T_x \Sigma$.  Show that $\widetilde{\alpha}$ vanishes at precisely the complex points of $\Sigma$.

\noindent {\bf Exercise}.  By a $C^{\infty}$-small perturbation of $\Sigma$, we can assume that $\Sigma$ has finitely many complex points [Hint: What are the dimensions of the Grassmanians $\text{Gr}_{\RR}(2,4)$ and $\text{Gr}_{\CC}(1,2)$?]

Recall the normal coordinates $(s,t)$ on $\Sigma \times D^2$.  Then the pair $ds, dt$ of 1-forms gives a coframing of $\Sigma$.  Define
\[\beta_1 \coloneqq \widetilde{ds} \qquad \beta_2 \coloneqq \widetilde{dt} \qquad \beta_3 = -\widetilde{ds} - \widetilde{dt}\]

\noindent {\bf Exercise} Show that $\beta_1 \wedge \beta_2 \neq 0$, except at the complex points of $\Sigma$.  In particular, $\beta_1$ vanishes at $x \in \Sigma$ if and only if $\beta_2$ vanishes at $x$.

\noindent {\bf Exercise} Suppose that $\beta_1,\beta_2$, viewed as section of $T^*\Sigma$, are transverse to the 0-section.  Show that at each complex point $x \in \Sigma$, the indices of the vanishing of $\beta_1$ and $\beta_2$ at $x$ agree.

\noindent {\bf Exercise} Show that $\beta_{\lambda}$ extends to a 1-form on the handlebody $H_{\lambda}$ of the trisection such that $\text{ker}(\beta_{\lambda})$ is the field of $J$-complex tangencies along $H_{\lambda}$.

Choose vector fields $\{v_1,v_2\}$ on $\Sigma$ such that
\begin{align*}
\beta_1(v_1) &= 0 & \beta_2(v_1) = \beta_1(v_2) & \geq 0 \\
\beta_2(v_2) &= 0
\end{align*}
and set $v_3 = -v_1 - v_2 \in \text{ker}(\beta_3)$.  Since $v_{\lambda} \in \text{ker}(\beta_{\lambda})$, we can extend $v_{\lambda}$ to a section of $\xi_{\lambda}$ over $H_{\lambda}$.

For notational purposes, let $\nu_{\lambda}$ be a normal vector fields to $H_{\lambda}$ such that near $\Sigma$, we have
\[u_1 = \del_s \qquad u_2 = \del_t \qquad u_3 = -\del_s - \del_t\]

{\bf Exercise} Show that the pairs
\[ \{u_1,v_1\} \qquad \{ u_2,v_2\} \qquad \{u_3,v_3\}\]
determine the same section of $\text{det}(TX,J)$ over $\Sigma$.

\begin{proposition}
Let $J$ be an almost-complex structure on the spine of a trisection $\cT$ of $X$.  Choose vector fields $\{v_{\lambda} \subset \xi_{\lambda}\}$ as above and let $\tau_{\lambda} = v_{\lambda}^{-1}(0)$.  The 1-complex
\[C_J = \tau_1 \cup \tau_2 \cup \tau_3\]
is the intersection of $PD(c_1(J))$ with the spine of the trisection $\cT$.
\end{proposition}

\begin{proof}
The bivector $u_{\lambda} \wedge v_{\lambda}$ determines a section of the determinant line bundle over $H_{\lambda}$.  The vector $u_{\lambda}$ is everywhere normal to $H_{\lambda}$ and nonvanishing, while $v_{\lambda}$ is tangent and vanishes along $\tau_{\lambda}$.  By the previous exercise, we obtain a section of the determinant bundle on the entire spine that vanishes precisely along the 1-complex $C_J$.  
\end{proof}

\vfill
\pagebreak

\section{$\SpinC$-structures}

A standard interpretation of a spin structure on a manifold $X$ is a trivialization of $TX$ over the 1-skeleton that extends across the 2-skeleton.  A similar interpretation of $\SpinC$-structures, due to Gompf, is an almost-complex structure over the 2-skeleton that extends across the 3-skeleton.

\subsection{Handle decompositions}
Every trisection $\cT$ of $X$ determines an {\it inside-out} handle decomposition as follows.  
\begin{enumerate}
\item Start with a neighborhood $\nu(\Sigma)$ of the central surface.  This is diffeomorphic to $\Sigma \times D^2$ and can be built in the standard way using a 0-handle, $2g$ 1-handles, and a 2-handle.  The boundary of this neighborhood is $\Sigma \times S^1$.  
\item Next, attach a neighborhood $\nu(H_{\lambda})$ of each 3-dimensional piece of the trisection.  The solid handlebody $H_{\lambda}$ is build from a single 0-handle and $g$ 1-handles.  Upside down, this becomes $g$ 2-handles and a single 3-handle.  Fix some distinct angular points $\theta_1,\theta_2,\theta_3 \in S^1$ in positive cyclic order.  Then attaching $\nu(H_{\lambda})$ is equivalent to the following.  Attach $g$ 2-handles along a cut system of curves on $\Sigma \times \{\theta_{\lambda}\}$ with surface framing.  After this surgery, the surface $\Sigma \times \{\theta_{\lambda}\}$ is now an essential 2-sphere and the 3-handle is attached along this 2-sphere.  The resulting boundary of the 4-manifold has three components $Y_1, Y_2, Y_3$ with $Y_3 \cong \#_{k_i} S^1 \times S^2$.
\item Finally, attach the 4-dimensional sectors.  These are 4-dimensional 1-handlebodies; upside down they consist of $k_i$ 3-handles and a single 4-handle.  The 3-handles are attached along the essential spheres in $\#_{k_i} S^1 \times S^2$.  The resulting boundary is three copies of $S^3$, which is where the 4-handles are attached.
\end{enumerate}
The {\it outside-in} handle decomposition determined by $\cT$ is the handle decomposition obtained by turning the inside-out handle decomposition upside down.

\subsection{Spin structures}

A standard interpretation of a spin structure on a manifold $X$ is a trivialization of $TX$ over the 1-skeleton that extends across the 2-skeleton.  Now, consider the inside-out handle decomposition of $X$ determined by a trisection $\cT$.  The 1-skeleton of $X$ is contained in the 1-skeleton of $\nu(\Sigma)$.  Thus, every spin structure of $X$ restricts to a spin structure on $\nu(\Sigma)$; moreover, since spin structures are stable, every spin structure of $X$ restricts to a spin structure on the central surface $\Sigma$.  

Recall that there exist two spin structures on $S^1$ and exactly one extends across $D^2$.  The spin structures on a closed, oriented surface $\Sigma$ are classified by maps
\[q: H_1(\Sigma; \ZZ/2\ZZ) \rightarrow \ZZ/2,\]
where $q(\gamma) = 0$ if the spin structure, restricted to a curve representing $\gamma$, is the spin structure that extends across the disk.  This map is a quadratic enhancement of the intersection form on $H_1(\Sigma)$; in particular, it satisfies the relation
\begin{equation}
\label{eq:quad-enhance}
q(x + y) = q(x) + q(y) + \langle x,y \rangle \, \text{mod $2$}.
\end{equation}
Let $\alphas = \{\alpha_i\}$ be a cut system of curves on $\Sigma$.  We say that $q(\alphas) = 0$ if $q(\alpha_i) = 0$ for every $\alpha_i \in \alphas$. Note that by the relation in Equation \ref{eq:quad-enhance}, if $q(\alphas) = 0$, then for every cut system $\alphas'$ obtained by handesliding some curves in $\alphas$, we also have $q(\alphas') = 0$.

\begin{proposition}
Let $\cT$ be a trisection of $X$ with trisection diagram $(\Sigma, \alphas,\betas,\gammas)$.  Then $X$ admits a spin structure if and only if there exists a quadratic enhancement $q: H_1(\Sigma;\ZZ/2\ZZ) \rightarrow \ZZ/2\ZZ$ such that
\[q(\alphas) = q(\betas) = q(\gammas) = 0.\]
Moreover, the set of spin structures is in 1-1 correspondence with such quadratic enhancements.
\end{proposition}

\begin{proof}
Each $q$ corresponds to a spin structure on $\Sigma$ and therefore a trivialization of $TX$ over its 1-skeleton.  In the inside-out handle decomposition, we have $3g + 1$ 2-handles.  One 2-handle corresponds to the 2-handle of $\Sigma$; by assumption the trivialization extends over this handle.  The remaining 2-handles are attached along the curves of $\alphas,\betas,\gammas$ with surface framing.  Consequently, the trivialization of $TX$ extends across such a handle if and only if the spin structure, restricted to the attaching circle, is the spin structure on $S^1$ that extends across the disk.
\end{proof}

\subsection{Lutz twists}  A Lutz twist is a method for modifying a 2-plane field $\xi$ along an embedded curve $\gamma$.

Fix a metric and orthonormal framing of $TH_{\lambda}$.  Let $\xi$ be a 2-plane field on $H_{\lambda}$.  Then $\xi$ determines a map $\psi: H_{\lambda} \rightarrow S^2$, by sending the unit normal vector to $\xi$ to its direction in $\RR^3$ using the framing of $TH_{\lambda}$.  Now let $\gamma$ be an embedded curve in $H_{\lambda}$.  The image $\psi(\gamma)$ is a closed loop $S^2$, which is contractible and therefore this path is homotopic to a constant path at the north pole.  Consequently, we can homotope $\xi$ and assume that $\psi(\gamma)$ is the constant map to the North pole.  Geometrically, this means that tangent vector $\gamma'$ is perpendicular to $\xi$ at every point along $\gamma$.

\begin{definition}
A {\it Lutz twist} of $\xi$ consists of the following operation.  Choose a framed neighborhood of $\gamma$, with coordinates $(r,\theta,t)$.  Assume that $\xi = \text{ker}(dt)$.  Now, choose smooth functions $f,g$ such that
\begin{enumerate}
\item $f: [0,2 \epsilon] \rightarrow \RR$, that is identically 0 near the endpoints and nonnegative.
\item $g:[0,2 \epsilon] \rightarrow \RR$, that is increasing; identically -1 near 0; identically 0 near $\epsilon$; and identically $1$ near $2 \epsilon$.
\end{enumerate}
Replace $\xi$ with
\[\widehat{\xi} = \text{ker}(g dt + f d\theta)\]
\end{definition}

{\bf Exercises}:

\begin{enumerate}
\item Show that applying two Lutz twists along $\gamma$ is homotopic to the identity.
\item We have described a {\it left-handed} Lutz twist -- i.e. the planes make a single left-handed turn along every diameter of the normal disk to $\gamma$.  We could alternatively do a {\it right-handed} Lutz twist by choosing $f$ to be nonpositive.  Show that left-handed and right-handed Lutz twists result in homotopic plane fields.
\end{enumerate}

A Lutz twist changes the relative Euler class of the plane field $\xi_{\lambda}$.  Let $\tau$ denote a fixed trivialization of $\xi_{\lambda}$ along $\Sigma$ and define the relative Euler class $e(\xi_{\lambda},\tau) \in H^2(H_{\lambda}, \Sigma) \cong H_1(H_{\lambda})$.

\begin{lemma}
For a Lutz twist along $\gamma$, the relative Euler classes satisfy
\[e(\xi,\tau) - e(\widehat{\xi},\tau) = 2 [\gamma] \in H_1(H_{\lambda})\]
\end{lemma}

\begin{proof}
We can extend $\tau$ to a framing that is $\{\del_r, \del_{\theta}\}$ in a tubular neighborhood of $\gamma$.  This framing must vanish along $\gamma$ and so $e(\xi_{\lambda},\tau) = A + [\gamma]$ for some $A \in H_1(H_{\lambda})$.  However, after the Lutz twist, we can use the same framing, which still vanishes along $\gamma$, except with opposite sign.  Thus $e(\widehat{\xi_{\lambda}},\tau) = A - [\gamma]$.
\end{proof}

\subsection{Action of $H^2(X;\ZZ)$}

The set of $\SpinC$-structures on $X$ is an affine copy of $H^2(X;\ZZ)$.  This means that $H^2(X;\ZZ)$ acts freely and transitively on the set of $\SpinC$-structures.  That is, given a $\SpinC$-structure $\s$ and some nonzero $A \in H^2(X;ZZ)$, there is a distinct $\SpinC$-structure $\s' = \s + A$.  Furthermore, the $1^{\text{st}}$-Chern classes satisfy
\[c_1(\s + A) = c_1(\s) + 2A\]

To describe the action of $H^2(X;ZZ)$ on the set of $\SpinC$-structures, we use the interpretation of $H^2(X)$ from Complex \ref{eq:hom-complex}.  Recall that we have a complex
\[
\xymatrix{
H^1(\Sigma) \ar[r] & \bigoplus_{\lambda} H_1(H_{\lambda}) \ar[r] & \bigoplus_{\lambda} H_1(Z_{\lambda}) 
}
\]
whose homology group is $H_2(X;\ZZ) \cong H^2(X;\ZZ)$.  In particular, the homology consists of triples $(a,b,c) \in  \bigoplus_{\lambda} H_1(H_{\lambda})$ such that
\[ a -b = 0 \in H_1(Z_1) \qquad b- c = 0 \in H_1(Z_2) \qquad c - a = 0 \in H_1(Z_3)\]
modulo the image of $H_1(\Sigma)$.

In order to move from almost-complex structures to $\SpinC$-structures, we need the following facts.

\begin{lemma}
Let $X$ be a closed 4-manifold with handle decomposition.  Let $J$ be an almost-complex structure on the 2-skeleton $X_2$ and let $\xi$ be the field of $J$-tangencies along the boundary $Y_2 \coloneqq \del X_2$.  In particular, $\xi$ is the 2-plane field $TY_2 \cap J(TY_2)$.  Then $J$ extends across a 3-handle attached along a 2-sphere $S \subset Y_2$ if and only if $\langle e(\xi), [S] \rangle = 0$.
\end{lemma}

\begin{proof}
One direction is obvious: if a 3-handle is attached along $S$ then $[S] = 0$ in $H_2(X;\ZZ)$.  Thus $\langle e(\xi),[S] \rangle = \langle c_1(J), [S] \rangle = 0$.  

Conversely, suppose that $\langle e(\xi), [S] \rangle = 0$.  There is a homotopy $\{\xi_t\}$ of 2-plane fields from $\xi = \xi_0$ to $\xi_1$ such that $\xi_1$ is the standard, {\it negative} tight contact structure in a neighborhood of $S$.  There is an almost-complex structure $J$ on $Y \times [0,1]$ whose restriction to $Y \times \{t\}$ is precisely $\xi_t$.   Finally, we can cap off with the Stein filling, which has a complex structure inducing $\xi_1$. 
\end{proof}

Choose a thickening of the spine and let $\{\widehat{Y}_{\lambda}\}$ denote its boundary components.  If $J$ is an almost-complex structure on the spine, let $\{\widehat{\xi}_{\lambda}\}$ denote the fields of $J$-complex tangencies.

\begin{corollary}
\label{cor:Spinc-trivial-euler}
An almost-complex structure $J$ on the spine of the trisection $\cT$ of $X$ is a $\SpinC$-structure if and only if the plane field $\widehat{\xi_{\lambda}}$ satisfies $e(\widehat{\xi}_{\lambda}) = 0$.
\end{corollary}

We can now define the action of $H^2(X;\ZZ)$ on a $\SpinC$-structure $\s$.
\begin{enumerate}
\item We can view $\s$ as an almost-complex structure on the spine such that the Euler classes $e(\widehat{\xi}_{\lambda})$ all vanish.
\item Given $A \in H^2(X;\ZZ)$, represent its Poincare dual in $H_2(X;\ZZ)$ by a triple $(a,b,c)$.  We can represent each element $a,b,c$, by an embedded collection of curves $\{\gamma_{\lambda} \subset H_{\lambda}\}$.
\item Modify $J$ by a Lutz twist on every component of $\gamma_{\lambda}$ for $\lambda = 1,2,3$.
\end{enumerate}

{\bf Exercise}: Show that after the Lutz twists, we still have that $e(\widehat{\xi}_{\lambda}) = 0$ for each $\lambda = 1,2,3$.

Consequently, the resulting almost-complex structure also extends across the 3-handles and determines a $\SpinC$-structure.

\vfill
\pagebreak

\bibliographystyle{alpha}
\nocite{*}
\bibliography{References}


\end{document}